\documentclass[12pt,a4]{article}
\usepackage{amsmath,amsthm,amsfonts,amssymb,amscd, amsxtra,color}
\usepackage[active]{srcltx}
\usepackage{url}
\usepackage{cite}
\usepackage{multirow}
\usepackage[margin=2.5cm,nohead]{geometry}
\theoremstyle{plain}
\theoremstyle{definition}
\newtheorem{theorem}{Theorem}
\newtheorem{lemma}{Lemma}
\newtheorem{definition}{Definition}
\newtheorem{corollary}{Corollary}
\newtheorem{proposition}{Proposition}

\newtheorem{example}{Example}
\newcommand{\beq}{\begin{equation}}
\newcommand{\eeq}{\end{equation}}
\newcommand{\tp}{^\top}

\def\min{\operatorname{min}}
\def\max{\operatorname{max}}

\DeclareMathOperator{\grad}{grad}

\DeclareMathOperator{\diag}{diag}
\DeclareMathOperator{\inte}{int}

\newcommand{\SP}{\mathbb S}
\newcommand{\R}{\mathbb R}

\begin{document}

\title{  On the spherical quasi-convexity of quadratic functions 
\thanks{This work was supported by CNxy (Grants  302473/2017-3 and 408151/2016-1) and FAPEG.}}

\author{
 O. P. Ferreira\thanks{IME/UFG, Avenida Esperan\c{c}a, s/n, Campus II,  Goi\^ania, GO - 74690-900, Brazil (E-mails: {\tt orizon@ufg.br}).}
\and
S. Z. N\'emeth \thanks{School of Mathematics, University of Birmingham, Watson Building, Edgbaston, Birmingham - B15 2TT, United Kingdom
(E-mail: {\tt s.nemeth@bham.ac.uk}, {\tt LXX490@student.bham.ac.uk}).}
\and
L. Xiao  \footnotemark[3]
}
\maketitle
\begin{abstract}
 In this paper the  spherical quasi-convexity of quadratic functions on spherically convex sets is studied. Several conditions characterizing  the  spherical quasi-convexity of quadratic functions  are presented. In particular, conditions implying spherical quasi-convexity of quadratic functions on the spherical positive orthant are  given. Some examples are provided as an application of the obtained results.
\end{abstract}
\noindent
{\bf Keywords:} Sphere, Spheric quasi-convexity \and  quadratic functions \and positive orthant.
\section{Introduction}
In this paper we study the spherical quasi-convexity of quadratic functions on spherically convex sets, which is related to the problem of finding
their minimum. This problem of minimizing a quadratic function on the sphere arised to S. Z. N\'emeth by trying to make certain fixed point theorems,
surjectivity theorems, and existence theorems for complementarity problems and variational inequalities more explicit; see
\cite{IsacNemeth2003,IsacNemeth2004,IsacNemeth2005,IsacNemeth2006,IsacNemeth2008}. In
particular, some existence theorems could be reduced to optimizing a quadratic function on the intersection of the sphere and a cone. Indeed, from
\cite[Corollary~8.1]{IsacNemeth2006} and \cite[Theorem18]{Nemeth2006} it follows immediately that for a closed convex cone $K\subseteq\R^n$ with
dual $K^*$, and a continuous mapping $F:\R^n\to\R^n$ with $G:\R^n\to\R^n$ defined by $G(x)=\|x\|^2F(x/\|x\|^2)$ and $G(0)=0$  differentiable at $0$,
if $\min_{\|u\|=1,u\in K}\langle DG(0)u, u\rangle>0$, where $DG(0)$ is the Jacobian matrix of $G$ at $0$, then the  nonlinear complementarity
problem defined by $K\ni x\perp F(x)\in K^*$ has a solution. Thus, we need to minimize a quadratic form on the intersection between a cone and the
sphere. These sets are exactly the spherically convex sets;  see \cite{FerreiraIusemNemeth2014}. Therefore, this leads to minimizing quadratic
functions on spherically convex sets.  In fact the optimization problem above reduces to the problem of calculating the scalar derivative,
introduced by S. Z. N\'emeth in  \cite{Nemeth1992,Nemeth1993,Nemeth1997},  along cones; see \cite{Nemeth2006}.    Similar minimizations  of
quadratic functions on spherically convex sets  are needed in the other settings; see \cite{IsacNemeth2003,IsacNemeth2004,IsacNemeth2005}.  Apart from the above, the motivation of this study is much wider. For instance, the quadratic constrained optimization problem
on the sphere
\begin{equation}\label{eq:gp}
	\min \{\langle Qx, x\rangle~:~x\in C\}, \qquad  C\subseteq\SP^{n-1}:=\left\{ x\in \R^{n}~:~\|x\|=1\right\}, 
\end{equation}
for a symmetric matrix $Q$, is a minimal eigenvalue problem in $C$, which includes the problem of finding the spectral norm of the matrix \(-Q\) when $C=\SP^{n-1}$ (see,
e.g.,~\cite{Smith1994}). It is  important to  highlight  that, the special case when $C$ is the nonnegative orthant is of particular interest because the nonnegativity of the minimal
value is equivalent to the copositivity of the matrix $Q$ \cite[Proposition~1.3]{UrrutySeeger2010} and to the nonnegativity of all Pareto 
eigenvalues of $Q$ \cite[Theorem~4.3]{UrrutySeeger2010}. As far as we are aware there are no methods for finding the Pareto spectra by using the
intrinsic geometrical properties of the sphere, hence our study is expected to open new perspectives for detecting the copositivity of a symmetric matrix.  The problem \eqref{eq:gp} also contains the trust region problem that appears in many nonlinear programming  algorithms
as a sub-problem, see \cite{DennisSchnabel1983, HagerPark2005, Hager2001}.
 
Optimization problems posed on the sphere, have a specific underlining algebraic structure that could be exploited to greatly reduce
the cost of obtaining the solutions; see \cite{Hager2001, HagerPark2005, Smith1994,So2011,Zhang2003,Zhang2012}. It is
worth to point out that when a quadratic function is spherically quasi-convex, then the spherical strict local minimum is equal to the spherical strict
global minimum.   Therefore, it is natural to
consider the problem of determining the spherically quasi-convex quadratic functions on spherically convex sets. The goal of the paper is to
present necessary conditions and sufficient conditions for quadratic functions which are spherically quasi-convex on spherical convex sets. As a
particular case, we exhibit several such results for the spherical positive orthant.  

The remainder of this paper is organized as follows. In Section~\ref{sec:int.1}, we recall some notations and basic results used throughout  the
paper. In Section~\ref{sec:cc} we present some general properties of spherically quasi-convex functions on spherically convex sets. In
Section~\ref{sec:qcqfcs}  we present  some conditions   characterizing   quadratic spherically quasi-convex functions on a general  spherically convex set.
In Section~\ref{sec:qcqfpo} we  present some  properties of  quadratic  function defined in the spherical positive orthant.  We conclude this paper by
making some final remarks in Section~\ref{sec;FinalRemarks}.

\section{Basics results} \label{sec:int.1}
In this section we present  the notations and the auxiliary results used throughout the paper. Let $\R^n$ be the $n$-dimensional Euclidean 
space with the canonical inner product $\langle\cdot,\cdot\rangle$ and  norm $\|\cdot\|$. The set of all $m
\times n$ matrices with real entries is denoted by $\R^{m \times n}$ and $\R^n\equiv \R^{n \times 1}$. Denote by $\R^n_+$ the nonnegative orthant and by 
$\R^n_{++}$ the positive orthant, that is, \[\R^n_+=\{x=(x_1,\dots,x_n)\tp:x_1\ge0,\dots,x_n\ge0\}\] and
\[\R^n_{++}=\{x=(x_1,\dots,x_n)\tp:x_1>0,\dots,x_n>0\}.\]
Denote by $e^i$ the $i$-th canonical unit vector in $\R^n$. A set $\cal{K}$ is called a {\it cone} if it is invariant under the multiplication with
positive scalars and a {\it convex cone} if it is a cone which is also a convex set. The {\it dual cone} of a cone ${\cal{K}} \subset \R^n$ is the cone 
${\cal{K}}^*\!\!:=\!\{ x\in \R^n : \langle x, y \rangle\!\geq\! 0, ~ \forall \, y\!\in\! {\cal{K}}\}.$   A cone \( {\cal K} \subset \R^{n}\)  is  
called {\it pointed} if \({\cal K} \cap \{-{\cal K}\} \subseteq \{0\}\), or equivalently,    if \({\cal K}\) does not contain straight  lines
through the origin.  Any pointed closed convex cone with nonempty interior will be called {\it proper cone}. The cone ${\cal K}$ is called {\it subdual} if
${\cal K}\subset {\cal K}^*$, {\it superdual} if ${\cal K}^*\subset {\cal K}$ and {\it self-dual} if ${\cal K}^*={\cal K}$. The matrix ${\rm I_n}$ denotes the
$n\times n$ identity matrix.    Recall that  $A=(a_{ij})\in \R^{n \times n}$ is {\it positive} if  $a_{ij}>0$ and {\it nonnegative}  if
$a_{ij}\geq 0$ for all $i,j=1, \ldots, n$.  A matrix   $A \in \R^{n \times n}$ is {\it reducible} if there is permutation matrix  $P \in \R^{n
\times n}$ so that 
$$
P^{T}AP=
\begin{bmatrix} 
B_{11} & B_{12} \\
0 & B_{22}
\end{bmatrix}, \quad   B_{11}\in \R^{m \times m}, ~B_{22}\in \R^{(n-m) \times (n-m)}, ~ B_{12}\in  \R^{m \times (n-m)},  \quad m<n.
$$
A matrix   $A \in \R^{n \times n}$ is {\it irreducible} if  it not reducible.  In the following  we state a version of   {\it Perron-Frobenius
theorem} for both positive matrices and nonnegative irreducible matrices,  its proof  can be found in \cite[Theorem~8.2.11,
p.500]{HornJohnson85} and  \cite[Theorem~8.4.4, p.508]{HornJohnson85}, respectively.

\begin{theorem} \label{Perron-Frobenius theorem1}
Let $A\in \R^{n \times n}$  be  either positive or nonnegative and  irreducible. Then,  $A$ has a dominant eigenvalue  $\lambda_{max}(A)\in \R$ with  associated eigenvector $v\in \R^n$ which satisfies the following properties:
\begin{itemize}
\item[i)] The  eigenvalue $\lambda_{max}(A)>0$ and its  associated eigenvector $v\in \R^n_{++}$;
\item[ii)] The eigenvalue $\lambda_{max}(A)>0$ has multiplicity one; 
\item[iii)] Every other  eigenvalue  $\lambda$ of $A$ is less that $\lambda_{max}(A)$ in absolute value, i.e, $|\lambda|<\lambda_{max}(A)$;
\item[iii)] There are no other positive  or  non-negative eigenvectors  of $A$ except positive multiples of $v$.
\end{itemize}
\end{theorem}
Remind  that   $A\in \R^{n \times n}$ is {\it copositive} if  $ \langle Ax, x \rangle\!\geq\! 0$ for all $x\in \R^n_+$  and  a {\it Z-matrix} is a matrix with nonpositive off-diagonal elements. Let ${\cal K} \subset \R^n $ be  a  pointed closed convex cone  with nonempty interior,  the  {\it ${\cal K}$-Z-property} of a matrix $A\in\R^{n\times n}$ means that 
$ \langle Ax,y\rangle \le0$, for any $(x,y)\in C({\cal K})$, where $C({\cal K}):=\{(x,y)\in\R^n\times\R^n:~x\in {\cal K},\textrm{ }y\in {\cal
K}^*,  \langle x, y \rangle=0\}$.  If $x=(x_1,\dots,x_n)\tp\in \R^n$, then $\diag (x)$ will denote an
$n\times n$  diagonal matrix with $(i,i)$-th entry equal to $x_i$, for $i=1,\dots,n$.  Throughout the paper the {\it \(n\)-dimensional  Euclidean sphere}  $\SP^{n-1}$ and its {\it tangent hyperplane at a point \(x\in \SP^{n-1}\)} are denoted, respectively, by
\[
\SP^{n-1}:=\left\{ x=(x_1, \ldots, x_{n})\tp \in \R^{n}~:~ \|x\|=1\right\}, \qquad T_{x}\SP^{n-1}:=\left\{v\in \R^n ~:~ \langle x, v \rangle=0 \right\},
\]
The  {\it intrinsic distance on the sphere} between two arbitrary points \(x, y \in \SP^{n-1}\)  is  defined by
\begin{equation} \label{eq:Intdist}
d(x, y):=\arccos \langle x , y\rangle.
\end{equation}
It can  be shown that   \((\SP^{n-1}, d)\) 
is a complete metric space, so that $d(x, y)\ge 0$ for all $x, y \in \SP^{n-1}$, and 
$d(x, y)=0$ if and only if $x=y$. It is also easy to check that $d(x, y)\leq \pi$ for all $x, y\in \SP^{n-1}$,
and $d(x, y)=\pi$ if and only if $x=-y$. The intersection curve of a plane though the origin of \(\R^{n}\) 
with the sphere \( \SP^{n-1}\) is called a { \it geodesic}.   If \(x, y\in \SP^{n-1}\) are such that \(y\neq  x\) and \(y\neq -x\), then the   unique {\it segment of  minimal  geodesic from to \(x\) to \(y\) } is 
\begin{equation} \label{eq:gctp}
\gamma_{xy}(t)= \left( \cos (td(x, y)) - \frac{\langle x, y\rangle \sin (td(x, y))}{\sqrt{1-\langle x, y\rangle^2}}\right) x 
+ \frac{\sin (td(x, y))}{\sqrt{1-\langle x, y\rangle^2}}\;y, \qquad t\in [0, \;1].
\end{equation}
Let \(x\in \SP^{n-1}\) and  \(v\in T_{x}\SP^{n-1}\) such that \(\|v\|=1\).  The  minimal segment of geodesic 
connecting  \(x\) to \(-x\),  starting at \(x\) with velocity \(v\)  at \(x\) is given by
\begin{equation} \label{eq:gctpp}
 \gamma_{x\{-x\}}(t):=\cos(t) \,x+ \sin(t)\, v, \qquad t\in  [0, \;\pi]. 
\end{equation}
Let \(\Omega \subset \SP^{n-1}\) be a spherically open set (i.e., a set open with respect to the induced topology in $\SP^{n-1}$). The {\it
gradient on the sphere} of a differentiable function \(f: \Omega \to  \R\) 
at a point \(x\in \Omega\) is the vector defined by
\begin{equation} \label{eq:grad}
\grad f(x):= \left[{\rm I_n}-xx^T \right] Df(x)= Df(x)- \langle Df(x) , x \rangle \,x,
\end{equation}
where \(Df(x) \in \R^{n}\) is the usual gradient of  \( f\) at \(x\in \Omega\).   Let ${\cal D}\subseteq\R^n$ be an open set, \(I\subset \R\) an open interval,  \(\Omega \subset
\SP^{n-1}\)  a spherically open set   
and   \(\gamma:I\to \Omega\) a  geodesic  segment. If \(f:{\cal D} \to \R\)  is a differentiable 
function, then,  since  \( \gamma'(t)\in T_{\gamma(t)} \SP^{n-1}\) for all \(t\in I\), the equality \eqref{eq:grad} implies 
\begin{equation} \label{eq:cr1}
\frac{d}{dt}f(\gamma(t)) =\left\langle \grad f(\gamma(t)), \gamma'(t) \right\rangle=
\left\langle D f(\gamma(t)), \gamma'(t) \right\rangle, \qquad \forall ~ t\in I.
\end{equation}
Let \( {\cal C}\subset\SP^{n-1}\) be a spherically convex set and $I\subset \R$ be  an interval.  
A function \(f:{\cal C} \to \R\)  is said to be spherically convex (respectively, strictly spherically convex) 
if for any minimal geodesic  segment \(\gamma:I\to {\cal C}\), the composition 
\( f\circ \gamma :I\to \R\) is convex (respectively, strictly  convex) in the usual sense. 

We end this section by stating some standard notations. 
We denote the {\it spherically open} and the {\it spherically closed ball} with radius \(\delta >0\) and center in 
\(x\in \SP^{n-1}\) by
$B_{\delta}(x):=\{y\in \SP^{n-1} : d(x, y)<\delta \}$ 
and
$\bar{B}_{\delta}(x):=\{y\in \SP^{n-1} : d(x, y)\leq \delta \}$,
respectively. The {\it sub-level sets} of a function  \(f:\R^n\supseteq{\cal M} \to \R\)   are denoted by  
\begin{equation} \label{eq:sls}
[f\leq c]:=\{x\in {\cal M} :\; f(x)\leq c\}, \qquad c\in \R.
\end{equation}

\section{Spherically quasi-convex functions   on spherically convex sets}\label{sec:cc}
In this section we study the general  properties of quasi-convex functions on the sphere. In particular,
we  present  the  first  order characterizations of differentiable quasi-convex functions on the sphere. Several results of this section have
already appeared in \cite{Nemeth1998}, but here these results have more explicit statements and proofs.   It is worth to remark that  the  quasi-convexity concept  generalize the convexity one, which was extensively studied in \cite{FerreiraIusemNemeth2014}.
\begin{definition}\label{def:cf}
The set \({\cal C} \subseteq \SP^{n-1}\) is said to be  \emph{spherically convex} if for any \(x\), \(y\in {\cal C} \) 
all the minimal geodesic segments joining \(x\) to \(y\)  are contained in \( {\cal C} \). 
\end{definition}

\begin{example} 
The set  \( S_{+}=\{(x_1,\dots, x_{n})\in \SP^{n-1}\, : \, x_{1}\geq 0,\dots, x_{n}\geq 0 \} \)    is  a closed spherically convex set.
\end{example}
We assume  for convenience that {\it from now on all spherically convex sets are nonempty proper 
subsets of the sphere}.  For each set \( A \subset \SP^{n-1}\), let \(K_A\subset\R^{n}\) be the {\it cone spanned by} \(A\), 
namely, 
\begin{equation} \label{eq:pccone}
K_A:=\left\{ tx \, :\, x\in A, \; t\in [0, +\infty) \right\}.
\end{equation}
Clearly, \(K_A\) is
the smallest cone which contains \(A\).  In the next proposition we exhibit a relationship of spherically convex sets with the cones  spanned by
them; for the  proof see  \cite{FerreiraIusemNemeth2013}.
\begin{proposition} \label{pr:ccs}
The set \({\cal C}\) is   spherically convex if and only if  the cone \(K_{\cal C}\) is convex and pointed.
\end{proposition}
Next we present a  generalization of the concept of convexity of a function on the sphere.
\begin{definition}\label{def:qcf-b}
Let \( {\cal C}\subset\SP^{n-1}\) be a spherically convex set and $I\subset \R$ be  an interval.  
A function \(f:{\cal C} \to \R\)  is said to be spherically quasi-convex (respectively, strictly spherically quasi-convex) 
if for any minimal geodesic  segment \(\gamma:I\to {\cal C}\), the composition 
\( f\circ \gamma :I\to \R\) is quasi-convex (respectively, strictly quasi-convex) in the usual sense, i.e., 
\(f(\gamma(t))\leq \max \{ f(\gamma(t_{1})), f(\gamma(t_{2}))\}\) for all \(t\in [t_{1}, t_{2}]\subset I\), (respectively,   \(f(\gamma(t))< \max \{ f(\gamma(t_{1})), f(\gamma(t_{2}))\}\)  for all \(t\in [t_{1}, t_{2}]\subset I\)).
\end{definition}
Naturally, from the above definition, it follows  that  spherically convex (respectively, strictly spherically convex) functions   are
spherically quasi-convex (respectively, strictly spherically quasi-convex), but the converse is not true; see \cite{FerreiraIusemNemeth2014}. 
\begin{proposition} \label{pr:charb1}
Let   \( {\cal C}\subset\SP^{n-1}\) be a spherically convex set.  A function  \(f:{\cal C} \to \R\) is   spherically quasi-convex if and only if   
the sub-level sets \([f\leq c]\) are   spherically convex  for all \(c\in \R\). 
\end{proposition}
\begin{proof}
Assume that   \(f\) is   spherically quasi-convex  and  \(c\in \R\). Take $x, y\in  [f\leq c]$ and $\gamma: [t_1, t_2]\to  \SP^{n-1}$ the  minimal  geodesic such that  \(\gamma(t_1)=x\) to \(\gamma(t_2)=y\),    see \eqref{eq:gctp} and \eqref{eq:gctpp}.  Since  \(f\) is a  spherically quasi-convex function and  $x, y\in  [f\leq c]$  we have
\(
f(\gamma(t))\leq \max \{ f(\gamma(t_{1})), f(\gamma(t_{2}))\}\leq c, 
\)
for all $t\in [t_{1}, t_{2}]$,  which implies that $\gamma(t)\in [f\leq c]$ for all \(t\in [t_{1}, t_{2}]\). Hence we conclude that $[f\leq c]$ is a spherically convex set,  for all \(c\in \R\).   Reciprocally,  we assume that \([f\leq c]\)  is    spherically convex set,  for all \(c\in \R\). Let \(\gamma:I\to {\cal C}\) be a minimal geodesic  segment. Take $t_{1}, t_{2} \in I$  and set $c= \max \{ f(\gamma(t_{1})), f(\gamma(t_{2}))\}$.  Since  \([f\leq c]\)  is    a spherically convex set,  we have $\gamma(t)\in [f\leq c]$ for all \(t\in [t_{1}, t_{2}]\), which implies 
\(
f(\gamma(t))\leq \max \{ f(\gamma(t_{1})), f(\gamma(t_{2}))\}, 
\)
for all $t\in [t_{1}, t_{2}]\subset I$. Therefore,  \(f\) is a  spherically quasi-convex function and the proof is concluded. 
\end{proof}

\begin{proposition} \label{pr:MinPro}
Let   \( {\cal C}\subset\SP^{n-1}\) be  spherically convex  and   \(f:{\cal C} \to \R\) be   spherically quasi-convex. If   $x^{*}\in {\cal C} $ is a strictly local minimum of $f$, then  $x^{*}$ is also a strictly  global minimum of $f$ in ${\cal C}$.
\end{proposition}
\begin{proof}
If  $x^{*}$ is a strictly local minimum of $f$, then there exists a number $\delta >0$ such that 
\begin{equation} \label{eq:lmg}
f(x)>f(x^{*}), \qquad \forall ~ x\in B_{\delta}(x^*)=\{y\in {\cal C} ~:~ d(y, x^*)<\delta\}.
\end{equation}
Assume by contradiction that $x^{*}$ is not a strictly  global minimum of $f$ in ${\cal C}$. Thus, there exists ${\bar x}\in {\cal C} $ with
${\bar x}\neq x^*$ such that  $f({\bar x})\leq f(x^{*})$. Since $C$ is  spherically convex, we can   take  \(\gamma:I\to {\cal C}\) a  minimal
geodesic  segment joining   $x^*$ and ${\bar x}$, let's say, $\gamma(t_1)=x^*$ and  $\gamma(t_2)={\bar x}$ with $t_1,t_2\in I$. Considering that
\(f\) is   spherically quasi-convex we have \(f(\gamma(t))\leq  \max \{ f(x^*), f({\bar x})\}=f(x^{*})\) for all \(t\in [t_{1}, t_{2}]\). On
the other hand, for $t$ sufficiently close to $t_{1}$ we have $\gamma(t)\in B_{\delta}(x^*)$. Therefore, the last inequality  contradicts
\eqref{eq:lmg} and the proof is concluded. 
\end{proof}

\begin{proposition} \label{pr:MinProS}
Let   \( {\cal C}\subset\SP^{n-1}\) be a spherically convex set and   \(f:{\cal C} \to \R\) be a  spherically  quasi-convex function. If  $f$ is strictly quasi-convex then   $f$ has at most one local  minimum point which   is also a  global  minimum point of $f$. 
\end{proposition}
\begin{proof}
Assume by contradiction that   $f$ has two local minimum  $x^*, {\bar x}\in {\cal C} $ with ${\bar x}\neq x^*$. Thus,   we can   take
\(\gamma:I\to {\cal C}\) a  minimal geodesic  segment joining   $x^*$ and ${\bar x}$, let's say, $\gamma(t_1)=x^*$ and  $\gamma(t_2)={\bar x}$
with $t_1,t_2\in I$. Due to  \(f\) being  strictly  spherically quasi-convex 
\(f(\gamma(t)) <  \max \{ f(x^*), f({\bar x})\}\) for all \(t\in [t_{1}, t_{2}]\).  Since we can take  $t$ sufficiently close to $t_{1}$ or
$t_{2}$, the last inequality contradict the assumption that    $x^*, {\bar x}$  are two distinct  local minima. Thus, $f$ has at most one local
minimum point.  Since  $f$ is strictly quasi-convex, the local minimum is strict. Therefore, Proposition~\ref{pr:MinPro} implies that  the  local
minimum point is   global and the proof is concluded. 
\end{proof}
\begin{proposition} \label{pr:CharDiff}
Let   \( {\cal C}\subset\SP^{n-1}\) be a open spherically convex set and   \(f:{\cal C} \to \R\) be a differentiable function. Then $f$ is spherically quasi-convex function if and only if
\begin{equation} \label{eq:cqqudf2}
f(x)\leq f(y) \Longrightarrow \langle Df(y) , x \rangle - \langle x , y \rangle\langle Df(y) , y \rangle \leq 0 , \qquad \forall ~ x,  y\in {\cal C} .
\end{equation}
\end{proposition}
\begin{proof}
Let  \(\gamma:I\to {\cal C}\)  be a geodesic  segment and consider  the composition \( f\circ \gamma :I\to \R\).    The usual characterization of scalar quasi-convex functions implies that \( f\circ \gamma\) is quasi-convex  if and  only if 
\begin{equation} \label{eq:cqqud}
 f(\gamma(t_1)\leq f(\gamma(t_2)) \Longrightarrow  \frac{d}{dt}\left(f(\gamma(t_2))\right)(t_1-t_2)\leq 0, \qquad  \forall ~ t_2, t_1 \in I.
\end{equation}
On the other hand,  for each  $ x,  y\in {\cal C}$ with $y\neq x$  we have from  \eqref{eq:gctp} that  $\gamma_{xy}$ is the  minimal  geodesic segment from \(x=\gamma_{xy}(0)\) to \(y=\gamma_{xy}(1)\) and 
$$
\gamma_{xy}'(1)=\displaystyle \frac{\arccos \langle x, y\rangle}{\sqrt{1-\langle x, y\rangle ^2}} 
\left(yy^T-{\rm I_n}\right)x\in T_{y}\SP^{n-1} , \qquad y\neq -x.
$$
Note that letting $x=\gamma(t_1)$ and  $y=\gamma(t_2)$ we have that $\gamma_{xy}(t)=\gamma(t_1+t(t_2-t_1))$. 
Therefore, by using \eqref{eq:cr1}  we conclude that the  condition in \eqref{eq:cqqud} is equivalent  to \eqref{eq:cqqudf2}  and the  proof  of the proposition  follows. 
\end{proof}
\section{Spherically quasi-convex quadratic functions  on spherically convex sets} \label{sec:qcqfcs}
In this section our aim is  to  present  some conditions   characterizing   quadratic  quasi-convex functions on a general  spherically convex set. For that we need some definitions:  {\it From now on we assume that  ${\cal K}\subset \R^{n}$ is   a proper subdual cone,  ${\cal C}=\SP^{n-1}\cap\inte({\cal K})$ is an  open spherically convex set and  $A=A^T\in\R^{n\times n}$   with  the associated quadratic function  $q_A: {\cal C}\to\R$ defined by }
\begin{equation} \label{eq:QuadFunc}
q_A(x):=\langle Ax,x\rangle.
\end{equation}
We also need the  {\it restriction  on $\inte {\cal K} $ of the Rayleigh quotient function}  $\varphi_A :  \inte {\cal K} \to \R$  defined by 
\begin{equation} \label{eq:RayleighFunction}
\varphi_A(x):=\frac{\langle Ax, x \rangle}{\|x\|^2}.
\end{equation}
In the following propositions we present some  equivalent characterizations of the convexity of $q_A$ defined by \eqref{eq:QuadFunc} on
spherically convex sets. Our first result is the following proposition.
\begin{proposition}\label{pr:spher-quasiconv}
Let $q_A$ and $\varphi_A$ be the functions defined in  \eqref{eq:QuadFunc} and  \eqref{eq:RayleighFunction}, respectively.  The following statements  are equivalents: 
	\begin{enumerate}
		\item[(a)] The quadratic function $q_A$ is spherically quasi-convex;
		\item [(b)] $\langle Ax,y\rangle\leq\langle x,y\rangle\max \left\{q_A(x), ~q_A(y)\right\}$ for all $ x,y\in\SP^{n-1}\cap {\cal K}$;
		\item [(c)] $ \displaystyle \frac{\langle Ax,y\rangle}{\langle x,y\rangle}\leq\max\left\{\varphi_A(x), ~\varphi_A(y)\right\}, $ for all $x,y\in {\cal K}$ with $\langle x,y\rangle\ne 0$.
	\end{enumerate}
\end{proposition}
\begin{proof}
 First of all, we  assume that item (a) holds.  Let $x,y\in {\cal C}$. Thus, either  $q_A(x)\leq q_A(y)$ or $q_A(y)\leq q_A(x)$. Hence,  by using Proposition~\ref{pr:CharDiff}  we conclude  that either $\langle Ay,x\rangle\leq\langle x,y\rangle q_A(y)$ or $\langle Ax,y\rangle\leq\langle x,y\rangle q_A(x)$. Thus,  since $A=A^T$ implies  $\langle Ax,y\rangle= \langle Ay,x\rangle$, taking into account that ${\cal K}$ is a  subdual cone, we have 
	 \[
	 \langle Ax,y\rangle\leq\max\{\langle x,y\rangle q_A(x),\langle x,y\rangle q_A(y)\}=\langle x,y\rangle\max\{q_A(x), q_A(y)\}, \qquad \forall ~  x,y\in {\cal C}. 
	 \]
Therefore, by continuity we extend the above inequality to all $ x,y\in\SP^{n-1}\cap {\cal K}$ and,  then item (b)  holds. Reciprocally, we assume that item (b) holds. Let $x,y\in {\cal C}$ satisfying   $q_A(x)\leq q_A(y)$. Then, by the inequality  in  item (b) and considering that ${\cal K}$ is a  subdual cone, we have
	$\langle Ax,y\rangle\leq\langle x,y\rangle q_A(y)$. Hence, by using  Proposition \ref{pr:CharDiff}  we conclude  that $f$ is spherically quasi-convex and the proof of  the equivalence between  (a) and (b) is complete.

To establish the equivalence between  (b) and (c), we assume first that  item (b) holds.  Let $x,y\in {\cal K}$ with $\langle x,y\rangle\ne 0$.  Then,  $x\ne 0$ and  $y\ne 0$. Moreover,  we  have  
$$
u:=\frac{x}{\|x\|} \in \SP^{n-1}\cap{\cal K},  \qquad v:=\frac{y}{\|y\|} \in \SP^{n-1}\cap {\cal K}.
$$ 
Hence, by using the inequality in item (b) with $x=u$ and $y=v$, we obtain the inequality in item (c).  Reciprocally, suppose that (c) holds.
Let $x,y\in \SP^{n-1}\cap {\cal K}$. First assume that  $\langle x,y\rangle\ne 0$. Since,   $\|x\|=\|y\|=1$, from the inequality in item (c)
we conclude that 
	$$
	\frac{\langle Ax,y\rangle}{\langle x,y\rangle}\leq\max\left\{ q_A(x), ~ q_A(y) \right\}. 
	$$
Due to ${\cal K}$  being  a  subdual cone,  $\langle x,y\rangle\geq 0$, and then the last inequality is equivalent to the inequality in   item
(b). Now, assume that $\langle x,y\rangle=0$. Then,  take  two  sequences  $\{x^k\}, \{y^k\}\subset  {\cal C}$ such that
$\lim_{k\to + \infty}x^k= x$ and $\lim_{k\to + \infty}y^k= y$.  Since    ${\cal K}$  is    a  subdual cone, we have $\langle x^k,y^k\rangle>0$
for all $k=1, 2, \ldots$. Therefore, considering that    $\|x^k\|=\|y^k\|=1$ for all $k=1, 2, \ldots$, we can apply again  the inequality in item
(c) to conclude 
$$
\langle A x^k,  y^k \rangle\leq\langle x^k, y^k\rangle\max \left\{q_A(x^k), ~ q_A(y^k)\right\}, \qquad k=1, 2, \ldots.
$$
By tending with $k$ to infinity, we  conclude  that the inequality in item (b) also holds for $\langle x,y\rangle=0$  and the proof of  the equivalence between  (b) and (c) is complete.
\end{proof}

\begin{corollary}\label{pr:qpcf}
Assume that ${\cal K}$ is  a  self-dual cone.  If the quadratic function $q_A$ is spherically quasi-convex, then $A$ has the $\cal K$-Z-property.
\end{corollary}
\begin{proof}
Let $x,y \in\R^n\times\R^n$  such that $x\in {\cal K}$, $y\in {\cal K}^*$ and $\langle x, y \rangle=0$. If $x=0$ or $y=0$ we  have  $\langle
Ax,y\rangle=0$. Thus, assume that  $x\neq 0$ and  $y\neq 0$. Considering that  ${\cal K}$ is  self-dual we have  $
x/\|x\|,~y/\|y\|\in\SP^{n-1}\cap {\cal K}$. Thus,   since   $q_A$ is spherically quasi-convex and   $\langle x/\|x\|, ~y/\|y\| \rangle=0$, we
obtain, from  items (a) and (b) of Proposition~\ref{pr:spher-quasiconv}, that  $\langle Ax,y\rangle\leq 0$. Therefore, $A$ has the $\cal K$-Z-property and 
the proof is concluded.
\end{proof}

\begin{theorem}\label{th:quasiconv-iff}
The function  $q_A$  defined in  \eqref{eq:QuadFunc} is spherically quasi-convex if and only if  $\varphi_A$  defined in  \eqref{eq:RayleighFunction} is quasi-convex.
\end{theorem}
\begin{proof}
Let   $[q_A\leq c]:=\left\{y\in {\cal C}~:~q_A(x)\leq c\right\}$ and $ [\varphi_A\leq c]:=\{x\in\inte(\mathcal K):\varphi_A(x)\leq c\}$  be the
sublevel sets of  $q_A$ and $\varphi_A$, respectively, where $c\in \R$.  Let ${\cal K}_{[q_A\leq c]}$ be   the cone spanned by $[q_A \leq c]$.
Since ${\cal C}=\SP^{n-1}\cap\inte({\cal K})$,  we conclude that  $x\in \inte \mathcal K $  if and only if $x/\|x\| \in {\cal C}$. Hence, the
definitions of $[q_A\leq c]$ and  $ [\varphi_A\leq c]$  imply  that  
\begin{equation}\label{eq:eqmain}
{\cal K}_{[q_A\leq c]}= [\varphi_A\leq c]. 
\end{equation} 
Now,  we assume that  $q_A$ is spherically quasi-convex. Thus, from Poposition~\ref{pr:charb1} we conclude that  $[q_A\leq c]$ is  spherically convex for all $c\in \R$.  Hence,  it follows from Proposition~\ref{pr:ccs} that the cone ${\cal K}_{[q_A\leq c]}$ is convex and pointed, which implies from \eqref{eq:eqmain}  that $[\varphi_A\leq c]$
 is convex for all $c\in \R$.  Therefore, $\varphi_A$ is quasi-convex.  Reciprocally, assume that $\varphi_A$ is quasi-convex. Thus,
 $[\varphi_A\leq c]$ is convex for all $c\in \R$. On the other hand, since ${\cal K}$  is  a proper subdual cone,    $\inte \mathcal K$ is
 pointed.  Thus, considering that  $[\varphi_A\leq c]\subset \inte \mathcal K$ is a cone,  it  is also  a  pointed cone. Hence, from
 \eqref{eq:eqmain}  it follows that ${\cal K}_{[q_A\leqq c]}$ is  a  pointed convex cone. Hence,  Proposition~\ref{pr:ccs} implies that $[q_A\leq c]$ is
 spherically convex for all $c\in \R$. Therefore, by using Proposition~\ref{pr:charb1}, we conclude that $q_A$ is spherically quasi-convex and the
 proof is completed. 
\end{proof}
\begin{corollary}\label{cor:cor}
The function  $q_A$  defined in  \eqref{eq:QuadFunc} is spherically quasi-convex  if and only if  the cone 
\begin{equation}\label{eq:closure}
\{x\in{\cal K}:\langle A_{c}x,x\rangle\leq 0\}, 
\end{equation}
is convex for any $c\in\R$, where $A_{c}:=A-c{\rm I_n}$.
\end{corollary}
\begin{proof}
Assume that $q_A$ is spherically quasi-convex. Hence   Theorem~\ref{th:quasiconv-iff} implies  that $\varphi_A$ is quasi-convex and then $[\varphi_A\leq c]$ is convex for any $c\in\R$. Since
	\begin{equation}\label{eq:closure_equiv}
		\{x\in{\cal K}~:~\langle A_{c}x,x\rangle\leq 0\}=\textrm{closure}(\{x\in\inte({\cal K})~:~\langle A_{c}x,x\rangle\leq 0\})
		=\textrm{closure}\left([\varphi_A\leq c]\right),
	\end{equation}
	where ``closure'' is the topological closure operator of a set, we conclude that the set in  \eqref{eq:closure} is convex. Reciprocally, assume that the set in  \eqref{eq:closure} is convex for any $c\in\R$. Thus, the equality in \eqref{eq:closure_equiv} implies that the set $\textrm{closure}([\varphi_A\leq c])$ is convex for any $c\in\R$. Hence,  $[\varphi_A\leq c]$ is convex  for any $c\in\R$ and then $\varphi_A$ is quasi-convex. Therefore, Theorem~\ref{th:quasiconv-iff} implies  that $q_A$ is spherically quasi-convex.
\end{proof}
\subsection{Spherically quasi-convex quadratic functions on the spherical positive orthant} \label{sec:qcqfpo}

In this section we present some  properties of a quadratic  function defined in the  spherical positive orthant, which corresponds to ${\cal
K}=\R^n_+$.  We know that if $A$ has only one
eigenvalue, then  $q_A$ is constant and, consequently,    it is   spherically quasi-convex. Therefore,  {\it  throughout  this section we  assume
that $A$  has at least two distinct  eigenvalues. The domains  ${\cal C}$  and $\inte({\cal K})$ of  $q_A$  and  $\varphi_A$,
respectively are given by  
 \begin{equation} \label{eq:KR++}
{\cal C}:=\SP^n\cap\R^n_{++}, \qquad \inte({\cal K}):=\R^n_{++},
\end{equation}
We remind that   $q_A$ and $\varphi_A$  are defined in  \eqref{eq:QuadFunc} and \eqref{eq:RayleighFunction}, respectively}. 
 Next we   present  a technical lemma which will be useful in the sequel. 
\begin{lemma}\label{Lem:Basic}
Let $n\geq  2$ and  $V=[v^1~ v^2~v^3~\cdots ~v^n] \in\R^{n\times n}$ be an orthogonal matrix, $A=V^\top\Lambda V$ and  $\Lambda=\diag(\lambda_1, \ldots, \lambda_n)$.  Assume that $\lambda_1 < \lambda_2 \leq  \ldots \leq \lambda_n$. If  $ v^1 \in \R^n_{+}$ and  $c\notin [\lambda_2,\lambda_n)$ then  the sublevel set  $[\varphi_A\leq c]$ is convex.
\end{lemma}
\begin{proof}
 By using that $VV^\top={\rm I_n}$ and $A=V^\top\Lambda V$  we obtain  from the definition \eqref{eq:RayleighFunction} that 
\begin{equation}\label{eq:bl1}
[\varphi_A\leq c]=\left\{x\in\R^n_{++}:~ \sum_{i=1}^n(\lambda_i-c)\langle v^i, x\rangle^2  \leq 0\right\}.
\end{equation}
We will show that $[\varphi_A\leq c]$ is convex,  for all $c\notin [\lambda_2,\lambda_n)$. If $c<\lambda_1$, then owing  that $v^1, v^2, \ldots ,
v^n$ are linearly independent, we conclude from \eqref{eq:bl1} that $[\varphi_A\leq c]=\{0\}$ and therefore it is convex.  If $c=\lambda_1$, then
from \eqref{eq:bl1} we conclude that $[\varphi_A\leq c]={\cal S}\cap \R^n_{++},$ where ${\cal S}: =\{x\in\R^n ~: \langle v^2, x\rangle=0,
~\ldots,~  \langle v^n, x\rangle=0\}$, and hence $[\varphi_A\leq c]$ is convex.  Assume that $\lambda_1<c<\lambda_2$.  By letting $y=V^\top x$, i.e., $y_i=\langle
v^i, x\rangle$, for $i=1, \ldots, n$,  and owing that  $v^1\in {\mathbb R^n_{++}}$ and $x\in\R^n_{++}$, we have $y_1>0$  and from \eqref{eq:bl1}
we obtain  that  $[\varphi_A\leq c]= {\cal L}\cap V\tp\R^n_{++}$, where  
$$
	{\cal L}:=\left\{y=(y_1,   \ldots,  y_n)\in\R^n:~y_1\geq  \sqrt{\theta_2y_2^2+ \ldots + \theta_{n}y_n^2}\right\},   \quad \theta_i=\frac{\lambda_i-c}{c-\lambda_1}, \qquad i=2, \ldots, n.
$$
Since   ${\cal L}$ and $V\tp\R^n_{++}$ are convex sets, we conclude that  $[\varphi_A\leq c]$ is  convex.   If $c\ge \lambda_n$, then
$[\varphi_A\leq c]=\R^n_{++}$ is convex, which concludes the proof.
\end{proof}
\begin{lemma} \label{lem:Copositive}
 Let  $\lambda $ be an  eigenvalues of $A$.  If  $\lambda  I_n-A$ is copositive and $\lambda\leq c$,   then  \[[\varphi_A\leq c]=\R^n_{++}\] and
 consequently it is  a   convex set.
\end{lemma}
\begin{proof}
Let $c\in \R$ and   $[\varphi_A\leq c]:=\{x\in  \R^n_{++}:~   \langle A x,x \rangle-c \|x\|^2\leq 0\}$.  Since $\lambda\leq c$ and $\lambda  I_n-A$ is copositive, we have  $\langle A x,x \rangle-c\|x\|^2\leq  \langle A x,x \rangle-\lambda \|x\|^2=  \langle (A-\lambda I_n) x,x \rangle\leq 0$ for all $x\in \R^n_{++}$, which implies that $[\varphi_A\leq c]=\R^n_{++}$.
\end{proof}
The next theorem exhibits a series of implications and, in particular, conditions which imply   that the quadratic function   $q_A$   is spherically 
quasi-convex.
\begin{theorem}\label{Thm:Copositive}
      Let $A\in\mathbb R^{n\times n}$ be a symmetric matrix,
       $\lambda_1\le\lambda_2\leq...\leq\lambda_n$ its eigenvalues.
	Consider the following statements:
      \begin{enumerate}
	       \item[(i)] $q_A$ is spherically quasiconvex.
	       \item[(ii)] $A$ is a Z-matrix.
              \item[(iii)] $\lambda_2I_n-A$ is copositive and
			there exist an eigenvector $v^1\in{\mathbb R}^n_+$
			of $A$ corresponding to the eigenvalue $\lambda_1$
			of $A$.
  	       \item[(iv)] $A$ is a Z-matrix and
			$\lambda_2\geq a_{ii}$ for
			any $i\in\{1,2,\ldots,n\}$.
              \item[(v)] $A$ is a Z-matrix $\lambda_1<\lambda_2$ and
			$\lambda_2\geq a_{ii}$ for any
			$i \in \{1,2,\ldots,n\}$.
              \item[(vi)] $A$ is an irreducible Z-matrix and
			$\lambda_2\geq a_{ii}$ for any
			$i\in\{1,2,\ldots,n\}$.
      \end{enumerate}
	Then, the following implications hold:

	\noindent\hspace{12mm}(v)

	\hspace{5.5mm} \noindent $\Downarrow$
	
	\noindent(iv)$\Leftarrow$(iii)$\Rightarrow$(i)$\Rightarrow$(ii)
	
	\hspace{5mm} \noindent $\Uparrow$
	
	\noindent\hspace{10.5mm}(vi)
	
\end{theorem}

\begin{proof}
$\,$
\vspace{2mm}

(v)$\Rightarrow$(iii)$\Leftarrow$(vi):
It is easy to verify that $\lambda_2I_n-A$ is nonnegative and hence copositive. Moreover, Perron-Frobenius theorem applied to the matrix $\lambda_2I_n-A$ implies that there exist an eigenvector $v^1\in\mathbb R^n_+$  corresponding to the largest eigenvalue $\lambda_2-\lambda_1$ of $\lambda_2I_n-A$, which is also the eigenvector of $A$ corresponding to $\lambda_1$.

(iii)$\implies$(i):  If $c\le\lambda_2$, then Lemma~\ref{Lem:Basic} implies that  $[\varphi_A\leq c]$ is convex.  If $c\geq \lambda_2$, then from Lemma~\ref{lem:Copositive} we have  $[\varphi_A\leq c]=\R^n_{++}$, which is  convex. Hence,  $[\varphi_A\leq c]$ is convex for all $c\in \R$. Therefore, by using Theorem~\ref{th:quasiconv-iff}, we conclude that  $q_A$ is spherically quasi-convex function.

(i)$\implies$(ii): From Corollary 1, it follows that $A$ has the $\R^n_+$-Z-property, which is easy to check that is equivalent to $A$ being a
Z-matrix. 

(iii)$\implies$(iv):  Since (iii)$\implies$(i)$\implies$(ii), it follows that $A$ is a Z-matrix. Since $\lambda_2I_n-A$ is copositive it follows that its diagonal elements are nonnegative. Hence, $\lambda_2\ge a_{ii}$ for any $i\in\{1,2,\ldots,n\}$.

\end{proof}

\begin{corollary}\label{Cor:NegMatrix}
 Let $n\ge 2$ and   $ \lambda_1, \ldots , \lambda_n \in \R$ be  the eigenvalues of $A$.   Assume that    $-A$ is  a positive matrix, $ \lambda_1 < \lambda_2 \leq  \ldots \leq \lambda_n$  and $0< \lambda_2$ . Then    $q_A$ is spherically quasi-convex.
\end{corollary}
\begin{proof}
First note that the matrix $\lambda_2I_n - A$ is a positive matrix and  $\lambda_2-\lambda_1>0$  is its  largest eigenvalue. Thus,
Theorem~\ref{Perron-Frobenius theorem1}   implies that the eigenvalue $\lambda_2-\lambda_1$  has the  associated eigenvector $v_1\in \R^n_{++}$.
Since $(\lambda_2I_n - A)v_1=(\lambda_2-\lambda_1)v_1$ we have $Av_1=\lambda_1v_1$. Hence   $ v_1$ is also an eigenvector associated to
$\lambda_1$. Therefore,  considering that $A$ is a Z-matrix, $ v^1 \in \R^n_{+}$, $\lambda_1 < \lambda_2$ and $\lambda_2\ge a_{ii}$ for any
$i\in\{1,2\dots,n\}$,  it follows from Theorem~\ref{Thm:Copositive} (v)$\Rightarrow$(i) that $q_A$ is spherically quasi-convex.
\end{proof}

In the following two examples   we use   Theorem~\ref{Thm:Copositive} (iii)$\Rightarrow$(i) to present a class of   quadratic quasi-convex functions defined  in the spherical positive orthant.
\begin{example}\label{ex:countergg}
Let $n\geq  3$ and  $V=[v^1~ v^2~v^3~\cdots ~v^n] \in\R^{n\times n}$ be an orthogonal matrix, $A=V^\top\Lambda V$ and  $\Lambda:=\diag(\lambda,\mu, \ldots,\mu ,\nu)$, where $\lambda,\mu,\nu\in\R$. Then  $q_A$ is a spherically quasi-convex, whenever 
\begin{equation} \label{eq:ciqc}
v^1-\sqrt{\frac{\nu-\mu}{\mu-\lambda}}|v^{n}|\in {\mathbb R^n_{+}},   \qquad  \lambda<\mu<\nu, 
\end{equation}
where $|v^{n}|:=(|v_1^{n}|, \ldots, |v_n^{n}|)$.  Indeed,  by using that $VV^\top={\rm I_n}$ and $A=V^\top\Lambda V$,  after   some calculations we conclude that  
\begin{equation}\label{eq:ceq}
  \langle A x,x \rangle-\mu\|x\|^2= (\mu-\lambda) \left[-\langle v^1, x\rangle^2 +\frac{\nu-\mu}{\mu-\lambda}\langle v^n, x\rangle^2\right].
\end{equation}
 Thus,  using the    condition in \eqref{eq:ciqc} and considering that $x\in\R^n_{++}$,  we have 
$$
-\langle v^1, x\rangle^2 +\frac{\nu-c}{c-\lambda}\langle v^n, x\rangle^2\leq -\langle v^1, x\rangle^2 +\frac{\nu-\mu}{\mu-\lambda}\langle v^n, x\rangle^2\leq\frac{\nu-\mu}{\mu-\lambda}\left[-\langle |v^n|, x\rangle^2 +\langle v^n, x\rangle^2\right] \leq 0.
$$
Hence, by combining the last inequality with     \eqref{eq:ceq}, we conclude  that  $\mu {\rm I_n}-A$ is copositive. Therefore,  since $ v^1 \in \R^n_{+}$  we can apply  
Theorem~\ref{Thm:Copositive} (iii)$\Rightarrow$(i) with $\lambda_2=\mu$ to conclude that $q_A$ is a spherically quasi-convex function. For instance,  taking  $\lambda<(\lambda+\nu)/2<\mu<\nu$ the vectors    $v^1=(e^1+e^n)/\sqrt{2}, v^2=e^2, ~ \ldots, ~ v^{n-1}=e^{n-1}, v^n=(e^1-e^n)/\sqrt2$,   satisfy \eqref{eq:ciqc}.
\end{example} 
\begin{example}\label{Ex:CountExMulEngfm}
Let $n\geq  3$ and  $V=[v^1~ v^2~v^3~\cdots ~v^n] \in\R^{n\times n}$ be an orthogonal matrix, $A=V^\top\Lambda V$ and  $\Lambda=\diag(\lambda_1, \ldots, \lambda_n)$. Then  $q_A$ is a spherically quasi-convex, whenever
\begin{equation} \label{eq:CondEngfm}
 v^1=(v^1_1, \ldots, v^1_n)\in {\mathbb R^n_{++}}, \qquad \lambda_1<\lambda_2\leq \cdots \leq \lambda_n \leq\lambda_2 +\frac{\alpha^2}{(n-2)}(\lambda_2-\lambda_1),
\end{equation}
where $\alpha:=\min\{ v^1_i \neq 0:~  i=1, \ldots, n\}$.
Indeed,  by using that $VV^\top={\rm I_n}$ and  the definition of the matrix $A$   in  \eqref{eq:CondEngfm} we  obtain that  
$$
  \langle A x,x \rangle-\lambda_{2}\|x\|^2= (\lambda_1-\lambda_{2})\langle v^{1}, x\rangle^2+  (\lambda_3-\lambda_{2})\langle v^{3}, x\rangle^2+ \cdots +(\lambda_n-\lambda_{2})\langle v^{n}, x\rangle^2.
$$
Since    $\lambda_{2}-\lambda_1>0$ and   $0\leq \lambda_{j}-\lambda_{2}\leq \lambda_n-\lambda_{2}$,  for all $j=3, \ldots, n$, the last equality becomes 
\begin{equation} \label{eq:coprfm}
  \langle A x,x \rangle-\lambda_{2}\|x\|^2 \leq (\lambda_{2}-\lambda_1)\left[-\langle v^1, x\rangle^2 +  \frac{\lambda_n-\lambda_{2}}{\lambda_{2}-\lambda_1}\left( \langle v^{3}, x\rangle^2+ \cdots +   \langle v^n, x\rangle^2\right)\right]. 
\end{equation}
On the other hand, by using  that  $v^1_i\in {\mathbb R_{++}}$ and  $v^1_i \geq \alpha$ for all   $i=1, \ldots, n$, we obtain that 
\begin{equation} \label{eq:CopFi}
 \langle v^1, x\rangle^2= (v^1_1 x_1+ \cdots + v^1_n x_n)^2\geq \alpha^2  ( x_1+ \cdots + x_n)^2\geq  \alpha^2  ( x_1^2+ \cdots + x_n^2)=\alpha^2\|x\|^2, 
\end{equation}
for all   $x\in\R^n_{+}$. Moreover,   taking into account that $\|v^j\|=1$,  for all $j=3, \ldots, n$, after some algebraic manipulation, it follows that 
$$
 \langle v^{3}, x\rangle^2+ \cdots +   \langle v^n, x\rangle^2\leq \| v^{3}\|^2\|x\|^2+\cdots+ \| v^{n}\|^2\|x\|^2 \leq (n-2)\|x\|^2,
$$
for all   $x\in\R^n_{+}$. Thus, combining the last inequalities with   \eqref{eq:coprfm}  and \eqref{eq:CopFi}  and considering  that  the last inequality in \eqref{eq:CondEngfm}  is equivalent to     $ -\alpha^2+ (n-2)(\lambda_n-\lambda_{2})/(\lambda_{2}-\lambda_1) \leq 0$, we have 
$$
\langle A x,x \rangle-\lambda_{2}\|x\|^2 \leq (\lambda_{2}-\lambda_1)\left[-\alpha^2 + (n-2) \frac{\lambda_n-\lambda_{2}}{\lambda_{2}-\lambda_1}\right]\|x\|^2\leq 0, 
$$
for all   $x\in\R^n_{+}$. Hence, we conclude  that  $\lambda_{2} {\rm I_n}-A$ is copositive. Therefore, since $v^1\in{\mathbb R}^n_+$ is the eigenvector  of $A$ corresponding to the eigenvalue $\lambda_1$,  we apply Theorem~\ref{Thm:Copositive} (iii)$\Rightarrow$(i), to conclude that $q_A$ is spherically quasi-convex function.    For instance,    $n\geq  3$,  $A=V^\top\Lambda V$,  $\Lambda=\diag(\lambda_1, \ldots, \lambda_n)$ and $V=[v^1~ v^2~v^3~\cdots ~v^n] \in\R^{n\times n}$, where $\alpha=1/\sqrt{n}$, 
$$
v^1:= \frac{1}{\sqrt{n}} \sum_{i=1}^ne^i, \quad   v^j:= \frac{1}{\sqrt{(n+1-j)+(n+1-j)^2}}\left[e^1- (n+1-j)e^j + \sum_{i>j}^ne^i \right],
$$
for $j=2, \ldots, n$ and   $\lambda_1<\lambda_2\leq \ldots \leq \lambda_n <\lambda_2 +(1/[n(n-2)])(\lambda_2-\lambda_1)$, satisfy the
orthogonality of $V$ and the condition \eqref{eq:CondEngfm}.
\end{example}  
In the next theorem we establish  the characterization for  quasi-convex quadratic functions $q_{A}$ on the spherical positive orthant where $A$  is symmetric having only two distinct eigenvalues. 
\begin{theorem}\label{thm:2neg-v}
Let $n\ge 3$ and $A\in\R^{n\times n}$ be a  symmetric matrix with only two distinct eigenvalues, such that its smallest one has multiplicity one.  Then, $q_{A}$ is spherically quasi-convex if and only if there is an eigenvector  of $A$ corresponding  to the smallest eigenvalue with all components nonnegative.
\end{theorem}
\begin{proof}
Let   $A:=(a_{ij})\in\R^{n\times n}$,   $\lambda_1, \lambda_2,\dots, \lambda_n$  be the  eigenvalues of $A$  corresponding to an orthonormal set
of eigenvectors  $v^1, v^2,\dots,v^n$, respectively.  Then, we  can assume without lose of generality that $\lambda_1=:\lambda <\mu:=\lambda_2=\dots=\lambda_n. $
Thus,  we have
\begin{equation} \label{eq:VLVT}
A=V\Lambda V^T, \qquad  V:=[v^1\textrm{ }v^2\textrm{  }\dots\textrm{ }v^n]\in\R^{n\times n}, \qquad  \Lambda:= \diag(\lambda,\mu, \ldots, \mu) \in\R^{n\times n}.
\end{equation}
First we assume that  $q_{A}$ is a spherically quasi-convex function. The matrix $\Lambda$ can be equivalently written as follows
\begin{equation} \label{eq:eqtr}
 \Lambda:=\mu {\rm I_n}+(\lambda-\mu)D,
\end{equation}
where $D:=(d_{ij})\in\R^{n\times n}$ has  all entries $0$ except the $d_{11}$ entry which is $1$.  Then \eqref{eq:eqtr} and \eqref{eq:VLVT} imply
\begin{equation} \label{eq:neg-v}
a_{ij}=(\lambda-\mu)v^1_iv^1_j\,  \qquad i\ne j.
\end{equation}
Since $q_{A}$ is spherically quasi-convex and $e^i \in {\cal C} $ for all $i=1, \ldots, n$, by  using item $(b)$ of
Proposition~\ref{pr:spher-quasiconv} we conclude that $a_{ij}\leq 0$ for all $i, j=1, \ldots, n$ with $i\ne j$. Thus,   owing  that  $\lambda <\mu
$, we obtain   form \eqref{eq:neg-v} that $0\leq v^1_iv^1_j $ for all $i\ne j$, which implies $v^1\in \R^n_+$ or  $-v^1\in \R^n_+$.  Therefore,  there is an eigenvector corresponding  to the smallest eigenvalue with all components nonnegative.   Reciprocally,  
 assume that $v^1\in \R^n_{+}$. Then, applying   Lemma~\ref{Lem:Basic}  with $ \lambda =\lambda_1<\mu=\lambda_2=\dots=\lambda_n$ we conclude that   $[\varphi_A\leq c]$ is
convex for any $c\in\R$, and then  $\varphi_A$ is quasi-convex. Therefore, by using Theorem~\ref{th:quasiconv-iff}, we conclude that $q_{A}$ is spherically quasi-convex function.
\end{proof}
In the following example we present a class  of matrices  satisfying the assumptions of Theorem~\ref{thm:2neg-v}.
\begin{example}
Let $v\in  \R^n_{+}$ and define the Householder matrix $H:={\rm I_n}-2vv^{T}/\|v\|^2$. The matrix  $H$ is nonsingular and    symmetric. Moreover,
$Hv=-v$ and letting $E:=\{ u\in \R^n~:~ \langle v, u \rangle=0\}$ we have $Hu=u$ for all $u\in E$.   Since the dimension of $E$ is $n-1$,   we
conclude that  $-1$ and $1$ are  eigenvalues of $H$  with multiplicities one and  $n-1$, respectively. Furthermore,   the eigenvector corresponding
to the smallest eigenvalue of $H$ has all components nonnegative. Therefore, Theorem~\ref{thm:2neg-v} implies that   $q_{H}(x)=\langle
Hx,x\rangle$ is spherically quasi-convex.
\end{example}
In order to  give a {\it complete characterization of the spherical quasiconvexity of $q_A$ for the case when $A$ is diagonal}, in the following result we 
start with a necessary condition for $q_A$ to be spherically quasi-convex on the spherical positive orthant.
\begin{lemma}\label{Lem:DiagPosCond}
Let $n\ge 3$, ${\cal C}=\SP^{n-1}\cap \R^n_{++}$ and $A\in\R^{n\times n}$ be a nonsingular   diagonal matrix.  If  $q_{A}$ is  spherically
quasi-convex, then  $A$ has only two distinct eigenvalues, such that its smallest one has multiplicity one.
\end{lemma}
\begin{proof}
The proof will be made by absurd.  First  we  assume  that  $A$ has  at least three distinct eigenvalues, among which exactly two are negative,  or  at
least two  distinct eigenvalues, among which exactly one is negative and has multiplicity greater than one, i.e., 
\begin{equation} \label{eq:DefVec}
Ae^1= -\lambda_1 e^1, \qquad   Ae^2= -\lambda_2 e^2,  \quad Ae^3= \lambda_3 e^3,  \qquad \lambda_1, \lambda_2, \lambda_3 >0
\end{equation}
with   $-\lambda_1<  -\lambda_2<0< \lambda_3$  or   $-\lambda_1=  -\lambda_2<0<\lambda_3$   and $e^1, e^2,e^3$ are  canonical vectors of $\R^n$. Define the following two  auxiliaries vectors 
\begin{equation} \label{eq:AuxVec}
v^1:=e^1+t_1e^3, \qquad  v^2:=e^2+t_2e^3, \qquad  t_i=\sqrt{\frac{\lambda_i}{\lambda_3}}, \qquad i=1,2.
\end{equation}
Hence, \eqref{eq:DefVec} and \eqref{eq:AuxVec} implies that $\langle Av^1, v^1\rangle =0$ and $\langle Av^2, v^2\rangle =0$ and owing that 
$v^1, v^2\in \R^n_{+}$, we  conclude that $v^1, v^2 \in \left\{x\in\R^n_{+}~:~\langle Ax, x\rangle \leq 0\right\}$.  However, using again \eqref{eq:DefVec} and \eqref{eq:AuxVec} we obtain that 
$$
\langle A(v^1+v_2), v^1+v_2 \rangle =2\sqrt{\lambda_1 \lambda_2} >0, 
$$
and then $v^1+v^2\notin  \left\{x\in\R^n_{+}~:~\langle Ax, x\rangle \leq 0\right\}= \textrm{closure}[q_A\leq 0]$. Thus,   the cone
$\textrm{closure}[q_A\leq 0]$  is not convex.  Finally,  assume by absurd that  $A$  has at least three  distinct eigenvalues or  at least
two distinct ones with the   smallest one  having  multiplicity greater  than one. Let  $\lambda, \mu, \nu$ be  eigenvalues of $A$ such  that
$\lambda <  \mu < \nu$ or $\lambda =  \mu < \nu$. Take  a constant $c\in \R$ such that  $\mu<c<\nu$.  Letting  $A_{c}:=A-c{\rm I_n}$ we conclude
that  $\lambda-c, \mu-c, \nu-c$ are   eigenvalues of $A_{c}$ and satisfy  $\lambda -c<\mu-c<0<\nu-c$ or $\lambda -c=\mu-c<0<\nu-c$.  Hence, by
the first part of the proof, with $A_c$ in the role of $A$, we conclude that $\textrm{closure}[q_{A_{c}}\leq 0]$  is not convex. 
Therefore,  Corollary~\ref{cor:cor} implies that $q_A$ is not spherically quasi-convex, which is absurd and the proof  is  complete. 
\end{proof}

The next result gives a full  characterization for $q_A$ to be spherically quasi-convex quadratic  function  on the spherical positive orthant,
where  {\it $A$  is a diagonal matrix}. The proof of  this result  is a combination of Theorem~\ref{thm:2neg-v}, Lemma~\ref{Lem:DiagPosCond} and
\cite[Theorem 1]{FerreiraNemeth2017}. Before presenting the result we need the following definition: A function is called \emph{merely spherically
quasi-convex} if it is spherically quasi-convex, but it is not spherically convex.
\begin{theorem}\label{cr:main}
Let $n\ge 3$ and $A\in\R^{n\times n}$ be a nonsingular   diagonal matrix. Then, $q_A$ is merely spherically quasi-convex if and only if  $A$ has
only two eigenvalues, such that its smallest one has multiplicity one and has a corresponding eigenvector with all components nonnegative.
\end{theorem}
We end this section by showing that, if a symmetric matrix $A$ has  three eigenvectors in the nonnegative  orthant  associated to at least two distinct 
eigenvalues, then the associated quadratic function  $q_A$  cannot  be   spherically quasi-convex.
\begin{lemma}\label{cor:DiagNegCond}  
Let $n\ge 3$ and   $v^1, v^2, v^3 \in \R^n$  be  distinct  eigenvectors   of a symmetric matrix $A$  associated  to the eigenvalues  $\lambda_1,
\lambda_2, \lambda_3 \in \R$, respectively, among which at least two are distinct.  If   $q_A$ is   spherically quasi-convex,  then $v^i \notin  \R^n_{+}$ 
for some $i=1,2,3$.
\end{lemma}
\begin{proof}
Assume by contradiction that   $v^1, v^2, v^3 \in \R^n_{+}$. Without loss of generality  we can also assume that $\|v^i\|=1$, for $i=1,2,3$.  We have three 
possibilities: $\lambda_1<  \lambda_2< \lambda_3$, $\lambda_1=  \lambda_2<\lambda_3$ or $\lambda_1<  \lambda_2=\lambda_3$.  We start by analyzing
the    possibilities $\lambda_1<  \lambda_2< \lambda_3$ or  $\lambda_1=  \lambda_2<\lambda_3$. First  we  assume  that   $\lambda_1<  \lambda_2<0<
\lambda_3$  or   $\lambda_1=  \lambda_2<0<\lambda_3$. Define the following   vectors 
\begin{equation} \label{eq:AuxVecg}
w^1:=v^1+t_1v^3, \qquad  w^2:=v^2+t_2v^3, \qquad  t_1=\sqrt{\frac{-\lambda_1}{\lambda_3}}, \qquad t_2=\sqrt{\frac{-\lambda_2}{\lambda_3}}.
\end{equation}
We have $\langle v^i, v^j\rangle =0$ for all  $i, j=1,2,3$ with $i\neq j$,  and 
\begin{equation} \label{eq:DefVecg}
Av^1= \lambda_1 v^1, \qquad   Av^2= \lambda_2 v^2,  \quad Av^3= \lambda_3 v^3,  \qquad v^1, v^2, v^3 \in \R^n_{+}, 
\end{equation} 
we conclude from  \eqref{eq:AuxVecg}  that $\langle Aw^1, w^1\rangle =0$ and $\langle Aw^2, w^2\rangle =0$.  Moreover, owing that $v^1, v^2, v^3 \in
\R^n_{+}$ we  conclude that $w^1, w^2 \in \left\{x\in\R^n_{+}~:~\langle Ax, x\rangle \leq 0\right\}$.  On the other hand, by using  \eqref{eq:DefVecg} and 
\eqref{eq:AuxVecg}, we obtain that 
$$
\langle A(w^1+w^2), w^1+w^2 \rangle =2\sqrt{\lambda_1 \lambda_2} >0, 
$$
and then $w^1+w^2\notin  \left\{x\in\R^n_{+}~:~\langle Ax, x\rangle \leq 0\right\}= \textrm{closure}[q_A\leq 0]$. Thus,   the cone
$\textrm{closure}[q_A\leq 0]$  is not convex.  Finally, for the general case,   take  a constant $c\in \R$ such that  $\lambda_2<c< \lambda_3$.
Letting  $A_{c}:=A-c{\rm I_n}$ we conclude that  $\lambda_1-c, \lambda_2-c, \lambda_3-c$ are   eigenvalues of $A_{c}$  and satisfying  $\lambda_1
-c<\lambda_2-c<0<\lambda_3-c$ or $\lambda_1 -c=\lambda_2-c<0<\lambda_3-c$ with the three corresponding   orthonormal eigenvectors $v^1, v^2, v^3
\in \R^n_{+}$.  Hence, by   the first part of the proof, with $A_c$ in the role of $A$, we conclude that the cone $\textrm{closure}[q_{A_{c}}\leq 0]$  is
not convex. Therefore,  Corollary~\ref{cor:cor} implies that $q_A$ is not spherically quasi-convex, which is a contradiction. To  analyze the possibility  
$\lambda_1<  \lambda_2=\lambda_3$, first assume that   $\lambda_1<0 <  \lambda_2= \lambda_3$ and define the vectors 
$$
w^1:=t_1v^1+v^3, \qquad  w^2:=t_2v^1+v^3, \qquad  t_1=\sqrt{\frac{\lambda_2}{-\lambda_1}}, \qquad t_2=\sqrt{\frac{\lambda_3}{-\lambda_1}}, 
$$
and  then proceed as above to obtain again a contradiction. Therefore, $v^i \notin  \R^n_{+}$ for some $i=1,2,3$ and the proof  is  complete. 
\end{proof}
\section{ Final remarks} \label{sec;FinalRemarks}
This paper is a continuation of   \cite{FerreiraIusemNemeth2013, FerreiraIusemNemeth2014, FerreiraNemeth2017}, where were studied  intrinsic
properties of the spherically  convex sets and functions.   As far as we know this  is the first study of spherically quasiconvex quadratic
functions. As an even more challenging problem, we will work towards developing efficient algorithms for  constrained optimization on spherically
convex sets. Minimizing a quadratic function on the spherical nonnegative orthant is of particular interest because the nonnegativity of the 
minimal value is equivalent to the copositivity of the corresponding matrix \cite[Proposition~1.3]{UrrutySeeger2010} and to the nonnegativity of 
its Pareto eigenvalues \cite[Theorem~4.3]{UrrutySeeger2010}. Considering the intrinsic geometrical properties of the sphere will open 
interesting perspectives for detecting the copositivity of a matrix. We foresee further progress in these topics in the nearby future. 
 \bibliographystyle{abbrv}
\bibliography{GenConvSphere}

\end{document}